\documentclass[leqno]{amsart}
\usepackage{amsthm}
\usepackage{color}
\usepackage{amssymb}
\usepackage{amsmath,amsfonts,enumerate}
\usepackage{amsthm}
\usepackage{enumerate}

\numberwithin{equation}{section}

\newtheorem{thm}{Theorem}
\newtheorem{lem}[thm]{Lemma}

\newtheorem{rmk}[thm]{Remark}

\newtheorem{definition}[thm]{Definition}

\newtheorem{problem}[thm]{Problem}

\newtheorem{proposition}[thm]{Proposition}

\begin{document}

\title{The optimal symmetric quasi-Banach range of the discrete Hilbert transform}

\author{K. Tulenov}
\address{
Al-Farabi Kazakh National University, 050040 Almaty, Kazakhstan;
 Institute of Mathematics and Mathematical Modeling, 050010 Almaty, Kazakhstan.}
\email{tulenov@math.kz}


\date{}
\begin{abstract} We identify symmetric quasi-Banach range of the discrete Calder\'{o}n operator and Hilbert transform acting on a symmetric quasi-Banach sequence space. As an application we present an example of optimal range in the case when the domain of those operators is the weak-$\ell_{1}$ space of sequences.
\end{abstract}

\maketitle

\section{Introduction}

In this paper, we consider symmetric sequence spaces (see \cite{LT,LSZ} and Definitions \ref{Sym} below).
Our study concerns the behavior of discrete Calder\'{o}n operator and Hilbert transform on such spaces and is motivated by the following connected problem.

\begin{problem}\label{dis case} Given a symmetric quasi-Banach sequence space $E=E(\mathbb{Z}),$ determine the least symmetric quasi-Banach sequence space $F=F(\mathbb{Z})$ such that
$\mathcal{H}^{d}(E)\subseteq F,$
where the Hilbert transform $\mathcal{H}^d$ is given by the formulae
$$(\mathcal{H}^{d}x)(n):=\frac{1}{\pi}\sum_{\substack{k\in\mathbb{Z}\\ k\neq n}}\frac{x(k)}{k-n},\quad x\in E(\mathbb{Z}).$$
\end{problem}

We shall be referring to the space $F(\mathbb{Z})$ as the optimal range space for the operator $\mathcal{H}^{d}$ restricted to the domain  $E(\mathbb{Z})\subset \ell_{\log}(\mathbb{Z}),$
where $\ell_{\log}(\mathbb{Z})$ is the Lorentz space of sequences associated with the function $\log(1+t),$ $t>0.$ Later we will show that the maximal symmetric domain for $\mathcal{H}^{d}$ is the space $\ell_{\log}(\mathbb{Z})$ (see Remark \ref{dom of Hilbert trans}).
The class of symmetric Banach spaces $E$ of sequences with Fatou norm (that is, when the norm closed unit ball $B_E$ of $E$ is closed in $E$ with respect to the almost everywhere convergence) such that $\mathcal{H}^{d}(E)\subseteq E$ was characterized by K. Anderson \cite[Theorem 3]{An} (see \cite[Theorem 2.1]{B} for the continuous case). A careful analysis of their proofs actually shows that the just cited result characterize the class of all symmetric Banach sequence spaces $E=E(\mathbb{Z})$ with the Fatou norm such that the optimal range for the operator $\mathcal{H}^{d}$ coincides with $E$. Nowadays, such a characterization is customarily stated in terms of the so-called Boyd indices  \cite{BSh,KPS,LT} and may be viewed as a satisfactory resolution of the Problem \ref{dis case} in the case of symmetric spaces $E$ possessing {\it non-trivial} Boyd indices.
However, the problem remains open if we deal with a symmetric space $E$ whose (at least) one Boyd index happens to be {\it trivial}.
If we restrict our attention to the subclass of symmetric {\it Banach} spaces $E$ with Fatou norm {\color{red} reduces to a familiar problem settled by D. Boyd} \cite{B} in 1967, and K. Anderson \cite{An} in 1976 for discrete case. Indeed, in this special case  \cite[Theorem 3]{An} asserts that $\mathcal{H}^d:E(\mathbb{Z})\to F(\mathbb{Z})$ if and only if $S^d:E(\mathbb{Z}_{+})\to F(\mathbb{Z}_{+}),$ where the operator $S^d,$ known as the Calder\'{o}n operator, is defined by the formula
$$(S^dx)(n)=\frac{1}{n+1}\sum_{k=0}^{n}x(k)+\sum_{k=n+1}^{\infty}\frac{x(k)}{k},\quad x\in \ell_{\log}(\mathbb{Z}_{+}).$$
Effectively, the problem reduces to describing the optimal receptacle of the operator $S^{d}.$

Consider the case when $E=\ell_p,$ $1<p<\infty.$ By Hardy's inequality, $S^{d}:\ell_p(\mathbb{Z}_{+})\to \ell_p(\mathbb{Z}_{+})$ and \cite[Theorem 3]{An} yields that $\mathcal{H}^{d}:\ell_p(\mathbb{Z})\to \ell_p(\mathbb{Z}).$ Now, if $E$ is a Banach interpolation space for the couple $(\ell_p,\ell_q),$ $1<p<q<\infty,$ then $\mathcal{H}^{d}:E(\mathbb{Z})\to E(\mathbb{Z}).$ A careful inspection of the proof in \cite{An} yields that, in this case, $E(\mathbb{Z})$ is the optimal receptacle for $\mathcal{H}^{d}$ restricted on $E(\mathbb{Z}).$ A similar assertion for the symmetric sequence spaces with Fatou norm can be extracted from \cite[Theorem 3]{An}.

In the case $p=1,$ Komori's result \cite{Ko} (Kolmogorov's result in the continuous case \cite{Kol}) asserts that the operator $\mathcal{H}^{d}$ maps the symmetric Banach sequence space $\ell_1(\mathbb{Z})$ to the symmetric {\it quasi-Banach} sequence space $\ell_{1,\infty}(\mathbb{Z}).$ From this, one might guess that, for $E=\ell_1(\mathbb{Z}),$ the optimal range of $\mathcal{H}^{d}$ is $\ell_{1,\infty}(\mathbb{Z}),$ which strongly suggests that the natural setting for Problem \ref{dis case} is that of symmetric quasi-Banach spaces. Addressing precisely this framework, one of our main results, Theorem \ref{quasi-banach opt range} provides a complete description of the optimal range $F$ for a given symmetric quasi-Banach space $E$ (under a mild technical assumption that $E(\mathbb{Z}_{+})\subset\ell_{\log}(\mathbb{Z}_{+}).$ It should be pointed out that our results allow a substantial extension of Anderson's and Komori's results cited above thereby complementing \cite[Theorem 3]{An} and \cite[Theorem]{Ko}.

The structure of this paper is as follows.
Section 2 recalls the necessary theory on symmetric quasi-Banach sequence spaces.
 In Section 3, we define the optimal range for the discrete Calder\'{o}n operator among the symmetric quasi-Banach sequence spaces.
  It is also shown the same result for the discrete Hilbert transform $\mathcal{H}^{d}$ in this section. Similar constructions were investigated in \cite{ST,ST1} for the Hardy and Hardy type operators in continuous case.
  As an application of the main result, we show an example of optimal range for the Calder\'{o}n operator and Hilbert transform when their domain is $\ell_{1,\infty}$ space of sequences. 

\section{Preliminaries}

\subsection{Symmetric quasi-Banach sequence spaces}

Let $I=\mathbb{Z}_+$ (resp. $\mathbb{Z}$), where $\mathbb{Z}_+:=\{0,1,2,...\}$ (resp. $\mathbb{Z}=\{...-2,-1,0,1,2,...\},$ and
let $(I,\nu)$ be a measure space equipped with counting measure $\nu.$ Let $L(I)$ be the space of finite $\nu$-a.e., measurable sequencess (with identification $\nu$ a.e.) on $I$. Define $L_{0}(I,\nu)$ (or $L_{0}(I)$) as the subalgebra of $L(I)$ which consists of all sequences $x$ such that $\nu(\{|x| > s\})$ is finite for
some $s.$
For $x\in L_{0}(I),$ we denote by $\mu(x)$ the decreasing rearrangement of the sequence $|x|,$ which is the sequence $|x|=\{|x(n)|\}_{n\geq 0}$ rearranged to be in decreasing order.
If $I = \mathbb{Z}_+$ (resp. $\mathbb{Z}$) and $\nu$ is the counting measure, then $L_{0}(I)=\ell_\infty(I)$, where $\ell_\infty(I)$ denotes the space of all bounded sequences on $I$. Briefly we denote $\ell_\infty(I)$ by $\ell_\infty.$

\begin{definition}\label{Sym} We say that $(E,\|\cdot\|_E)$ is a symmetric (quasi-)Banach sequence space on $I$ if the following hold:
\begin{enumerate}[{\rm (i)}]
\item $E$ is a subset of $\ell_\infty;$
\item $(E,\|\cdot\|_E)$ is a (quasi-)Banach space;
\item If $x\in E$ and $y\in \ell_\infty$  are such that $|y|\leq|x|,$ then $y\in E$ and $\|y\|_E\leq\|x\|_E;$
\item If $x\in E$ and $y\in \ell_\infty$ are such that $\mu(y)=\mu(x),$ then $y\in E$ and $\|y\|_E=\|x\|_E.$
\end{enumerate}
\end{definition}
For the general theory of symmetric spaces of functions and sequences, we refer the reader to \cite{BSh,KPS,LT}.

The discrete dilation operators $\sigma_m,$ $m\in\mathbb{N},$ on $\ell_{\infty}(\mathbb{Z}_{+})$ is defined by
$$\sigma_{m}(x):=\{\underbrace{x(0),x(0),...,x(0)}_{m \,\ \text{terms}},\underbrace{x(1),x(1),...,x(1)}_{m \,\ \text{terms}},...,\underbrace{x(n),x(n),...,x(n)}_{m \,\ \text{terms}},...\}.$$
It is obvious that the dilation operator $\sigma_{m}$ is continuous in $\ell_{\infty}(\mathbb{Z}_{+})$ (see \cite[Chapter II.3, p. 96]{KPS}).

\subsection{Lorentz sequence spaces}\label{S Lorentz}

Let $\varphi$  be an increasing concave function on $\mathbb{R}_{+}$ such that $\lim_{t\rightarrow 0+}\varphi(t)=0.$
 The Lorentz sequence space $\ell_{\varphi}(I)$ is defined by setting
$$\ell_{\varphi}(I):=\left\{x\in \ell_{\infty}(I): \|x\|_{\ell_{\varphi}}=\sum_{n=0}^{\infty}\mu(n,x)(\varphi(n+1)-\varphi(n))<\infty\right\},$$
and equipped with the norm
\begin{equation}\label{Lorsec}
\|x\|_{\ell_{\varphi}}=\sum_{n=0}^{\infty}\mu(n,x)(\varphi(n+1)-\varphi(n)).
\end{equation}
These spaces are examples of symmetric Banach sequence spaces.
In particular, if $\varphi(t):=\log(1+t),\,\ t>0,$ then the Lorentz sequence space $\ell_{\log}(I)$ is defined as follows:
$$\ell_{\log}(I):=\left\{x\in \ell_{\infty}(I):\|x\|_{\ell_{\log}}=\sum_{n=0}^{\infty}\frac{\mu(n,x)}{n+1}<\infty\right\}.$$
For more details on Lorentz spaces, we refer the reader to \cite[Chapter II.5]{BSh} and \cite[Chapter II.5]{KPS}.

\subsection{Weak $\ell_{1}$ and $m_{1,\infty}$ spaces}\label{weak L1}

The weak-$\ell_1$ sequence space $\ell_{1,\infty}$ on $I$ is defined as
$$\ell_{1,\infty}(I):=\left\{x\in \ell_{\infty}(I):\mu(n,x)=O\left(\frac{1}{1+n}\right)\right\},$$
equipped with the functional $\|\cdot\|_{\ell_{1,\infty}}$ defined by the formula

$$\|x\|_{\ell_{1,\infty}}:=\sup_{n\in\mathbb{Z}_{+}}(n+1)\mu(n,x).$$
It is easy to see that $\|\cdot\|_{\ell_{1,\infty}}$ is a quasi-norm. In fact, the space $(\ell_{1,\infty},\|\cdot\|_{\ell_{1,\infty}})$ is quasi-Banach.
(see \cite[Example 1.2.6, p. 24]{LSZ}).

Define the Marcinkiewicz (or Lorentz) space $m_{1,\infty}(I)$ by setting
\begin{equation}\label{Marsecuence}
m_{1,\infty}(I):=\left\{x\in \ell_{\infty}(I): \sup_{n\in\mathbb{Z}_{+}}\frac{1}{\log(2+n)}\sum_{k=0}^{n}\mu(k,x)<\infty\right\}
\end{equation}
equipped with the norm
$$\|x\|_{m_{1,\infty}}:=\sup_{n\in\mathbb{Z}_{+}}\frac{1}{\log(2+n)}\sum_{k=0}^{n}\mu(k,x)
$$
(see \cite[Example 1.2.7, p. 24]{LSZ}).
This is an example of a symmetric Banach sequence space. For more information on Marcinkiewicz spaces we refer to \cite[Chapter II.5]{BSh} and \cite[Chapter II.5]{KPS}.

Moreover, the space $m_{1,\infty}$ contains the quasi-Banach space $\ell_{1,\infty}$ of sequences, i.e. the following inclusion
$$\ell_{1,\infty}\subset m_{1,\infty}$$
is strict (see \cite[Lemma 1.2.8 and Example 1.2.9, pp. 25-26]{LSZ}).

\subsection{Discrete Calder\'{o}n operator and Hilbert transform}\label{calderon}
Define the space $(\ell_{1,\infty}+\ell_{\infty})(I)$ as follows:
$$(\ell_{1,\infty}+\ell_{\infty})(I):=\{x=x_{1}+x_{2}:x_{1}\in\ell_{1,\infty}(I),x_{2}\in\ell_{\infty}(I),\|x\|_{\ell_{1,\infty}+\ell_{\infty}}<\infty\},$$
where the quasi-norm is defined by
$$\|x\|_{\ell_{1,\infty}+\ell_{\infty}}:=\inf\{\|x_{1}\|_{\ell_{1,\infty}}+\|x_{2}\|_{\ell_{\infty}}:x_{1}\in\ell_{1,\infty}(I),x_{2}\in\ell_{\infty}(I)\}.$$
For each $x\in \ell_{\log}(\mathbb{Z}_{+}),$ define the discrete Calder\'{o}n operator $S:\ell_{\log}(\mathbb{Z}_{+})\rightarrow(\ell_{1,\infty}+\ell_{\infty})(\mathbb{Z}_{+})$ by
\begin{equation}\label{S dis}(S^{d}x)(n):=\frac{1}{n+1}\sum_{k=0}^{n}x(k)+\sum_{k=n+1}^{+\infty}\frac{x(k)}{k}, \,\,\ x\in \ell_{\log}(\mathbb{Z}_{+}).
\end{equation}
It is obvious that $S^{d}$ is a linear operator.
Next, it is easy to see that if $0<n<n_{0},$ then
$$\min\left(1,\frac{m}{n_{0}}\right)\leq \min\left(1,\frac{m}{n}\right)\leq \frac{n_{0}}{n}\cdot\min\left(1,\frac{m}{n_{0}}\right), \,\, (m>0).$$
So, if $x$ is nonnegative, it follows from the first of these inequalities that $(S^{d}x)(n)$ is a
decreasing function of $n\geq0.$ The operator $S^{d}$ is often applied to the decreasing rearrangement $\mu(x)$ of a function $x$ defined on some other measure space.
Since $S^{d}\mu(x)$ is itself decreasing, it is easy to see that $\mu(S^{d}\mu(x))=S^{d}\mu(x).$
Let $x\in \ell_{\log}(\mathbb{Z}_{+}).$ Since for each $n\geq0,$ the kernel
$K_{n}(k)=\frac{1}{k}\cdot\min\Big\{1,\frac{k}{n+1}\Big\}$
is a decreasing sequence of $k> 0,$ it follows  from \cite[Theorem II.2.2, p. 44]{BSh} that
\begin{equation}\label{S proper}\begin{split}
|(S^{d}x)(n)|
&\stackrel{\eqref{S dis}}{=}\Big|\sum_{k=0}^{\infty}x(k)\min\Big\{1,\frac{k}{n+1}\Big\}\frac{1}{k}\Big|\\
&\leq\sum_{k=0}^{\infty}|x(k)|\min\Big\{1,\frac{k}{n+1}\Big\}\frac{1}{k}\leq\sum_{k=0}^{\infty}\mu(k,x)\min\Big\{1,\frac{k}{n+1}\Big\}\frac{1}{k}\\
&\stackrel{\eqref{S dis}}{=}(S^{d}\mu(x))(n), \quad \forall n\in \mathbb{Z}_{+}.
\end{split}\end{equation}
For more information about these operators, we refer to \cite[Chapter III]{BSh} and \cite[Chapter II]{KPS}.

If $x\in \ell_{\log}(\mathbb{Z}),$ then the discrete Hilbert transform $\mathcal{H}^{d}$ is defined as in \cite{An} by
\begin{equation}\label{hilbert tr}
(\mathcal{H}^{d}x)(n):=\frac{1}{\pi}\sum_{\substack{k\in\mathbb{Z}\\ k\neq n}}\frac{x(k)}{n-k},\quad x\in\ell_{\log}(\mathbb{Z}).
\end{equation}

\begin{rmk}\label{dom of Hilbert trans} Let $x=x\chi_{[0,\infty)}$ such that $x$ is a non-negative decreasing sequence on $\mathbb{Z}_{+}.$ Then it is easy to see that
\begin{eqnarray*}\begin{split} |(\mathcal{H}^{d}x)(-n)|
&\stackrel{\eqref{hilbert tr}}{=}\frac{1}{\pi}\left|\sum_{k=0}^{\infty}\frac{x(k)}{-n-k}\right|\\
&=\frac{1}{\pi}\sum_{k=0}^\infty\frac{x(k)}{n+k}=\frac{1}{\pi}\left(\sum_{k=0}^{n}\frac{x(k)}{n+k}+\sum_{k=n+1}^{\infty}\frac{x(k)}{n+k}\right)\\
&\geq\frac{1}{\pi}\left(\sum_{k=0}^{n}\frac{x(k)}{2(n+1)}+\sum_{k=n+1}^{\infty}\frac{x(k)}{2k}\right)\\
&=\frac{1}{2\pi}\cdot\left(\frac{1}{n+1}\sum_{k=0}^{n}x(k)+\sum_{k=n+1}^{\infty}\frac{x(k)}{k}\right)\stackrel{\eqref{S dis}}{=}\frac{1}{2\pi}(S^dx)(n), \,\, n\in\mathbb{Z}_{+},
\end{split}\end{eqnarray*}
i.e. we have
$$\frac{1}{2\pi}(S^{d}\mu(x))(n)\leq |(\mathcal{H}^{d}x)(-n)|, \,\, n\in \mathbb{Z}_{+}.$$
Therefore, if $(\mathcal{H}^{d}x)(-n)$ exists, then it follows that $S^{d}\mu(x)$ exists, and it means $x$ belongs to the domain of $S^{d},$ i.e. $x\in \ell_{\log}(\mathbb{Z}_{+})$ (see \eqref{S dis}).
On the other hand, if $x\in\ell_{\log}(\mathbb{Z}_{+}),$ then by \cite[Theorem III.4.8, p. 138]{BSh}, we have
$$\mu(\mathcal{H}^{d}x)\leq c_{abs} S^{d}\mu(x),$$
which shows existence of $\mathcal{H}^{d}x.$
\end{rmk}

\section{Optimal symmetric quasi-Banach range for the discrete Calder\'{o}n operator and Hilbert transform}

In this section, we describe the optimal symmetric quasi-Banach sequence range spaces for the discrete Calder\'{o}n operator $S^{d}$  and Hilbert transform $\mathcal{H}^{d}$. A continuous version of the following lemma was proved in \cite{STZ}, here we use the discrete version of it.
\begin{lem}\label{measure conv}
Let $\{x_k\}_{k=1}^{\infty}\subset \ell_{\infty}(\mathbb{Z}_{+}),$ where $x_k=\{x_{k}(n)\}_{n\in\mathbb{Z}_{+}}.$
If the series
$$\sum_{k=1}^{\infty}\sigma_{2^{k}}\mu\big(n,x_{k}\big)$$
converges almost everywhere (a.e.) in $\ell_{\infty}(\mathbb{Z}_{+}),$
for all $n\in\mathbb{Z}_{+},$
then the series
$\sum_{k=1}^{\infty}x_{k}$ converges in measure in $\ell_{\infty}(\mathbb{Z}_{+})$
and we have
\begin{equation}\label{est conv}
\mu\Big(n,\sum_{k=1}^{\infty}x_k\Big)\leq\sum_{k=1}^{\infty}\sigma_{2^{k}}\mu(n,x_k), \,\ n\in\mathbb{Z}_{+}.
\end{equation}
\end{lem}
\begin{proof} Fix $\varepsilon,\delta>0,$ and choose $N=N(\varepsilon,\delta)$ such that
\begin{equation}\label{epsilon delta est}
\left(\sum_{k=N}^{\infty}\sigma_{2^{k}}\mu(x_k)\right)(\varepsilon)<\delta.
\end{equation}
Then, for any $N_1,N_2\geq N$ and by (2.23) in \cite[Corollary II.2, p. 67]{KPS} and \eqref{epsilon delta est}, we have
\begin{equation}\label{partial sum est}
\begin{split}\mu\left(\varepsilon,\sum_{k=N_{1}+1}^{N_2}x_{k}\right)
&=\mu\left(\varepsilon\cdot\frac{\sum_{k=N_{1}+1}^{N_2}2^{-k}}{\sum_{m=N_{1}+1}^{N_2}2^{-m}},\sum_{n=N_{1}+1}^{N_2}x_{k}\right)\\
&\stackrel{(2.23)}\leq\sum_{k=N_{1}+1}^{N_2}\mu\left(\varepsilon\cdot\frac{2^{-k}}{\sum_{m=N_{1}+1}^{N_2}2^{-m}},x_{k}\right)\\
&\leq\sum_{k=N_{1}+1}^{N_2}\mu(\varepsilon\cdot2^{N_{1}-k},x_{k})\\
&\leq\sum_{k=N}^{\infty}\mu(\varepsilon\cdot2^{-k},x_{k})=\left(\sum_{k=N}^{\infty}\sigma_{2^{k}}\mu(x_k)\right)(\varepsilon)\stackrel{\eqref{epsilon delta est}}<\delta.
\end{split}
\end{equation}
Let us denote $a_{N_1}(n)=\sum_{k=1}^{N_1}x_{k}(n)$ and  $a_{N_2}(n)=\sum_{k=1}^{N_2}x_{k}$ for all $n\in\mathbb{Z}_{+}.$
Then, by the preceding inequality for any $N_{1},N_{2}\geq N,$ we obtain
$$a_{N_2}(n)-a_{N_{1}}(n)\in U(\varepsilon,\delta):=\{x\in \ell_{\infty}(\mathbb{Z}_{+}): m(\{|x|>\delta\})<\varepsilon\},$$
which shows that $\{a_k\}_{k=1}^{\infty}$ is a Cauchy sequence in measure in $\ell_{\infty}(\mathbb{Z}_{+}).$
 Since $\ell_{\infty}(\mathbb{Z}_{+})$ is complete in measure topology, it follows that
the series $\sum_{k=1}^{\infty}x_{k}(n)$ converges in measure for all $n\in\mathbb{Z}_{+}.$
Therefore, since the decreasing rearrangement $\mu$ is continuous from the right, it follows from \eqref{partial sum est} that
$$\mu\left(n,\sum_{k=1}^{\infty}x_k\right)\leq\sum_{k=1}^{\infty}\sigma_{2^{k}}\mu(n,x_k), \,\ n\in\mathbb{Z}_{+}.$$
\end{proof}

\begin{definition}\label{discrete quasi-banach range} Let $E$ be a quasi-Banach symmetric sequence space on $\mathbb{Z}_{+}.$ Let $E(\mathbb{Z}_{+})\subset \ell_{\log}(\mathbb{Z}_{+})$ and let
$S^{d}$ be the operator defined in \eqref{S dis}. Define
$$F(\mathbb{Z}_{+}):=\{x\in(\ell_{1,\infty}+\ell_{\infty})(\mathbb{Z}_{+}):  \exists y\in E(\mathbb{Z}_{+}), \, \mu(x)\leq S^{d}\mu(y)\}$$
such that
\begin{equation}\label{opt norm}\|x\|_{F(\mathbb{Z}_{+})}:=\inf\{\|y\|_{E(\mathbb{Z}_{+})}:\mu(x)\leq S^{d}\mu(y)\}<\infty.
\end{equation}
\end{definition}

The following is the main result of this section.
\begin{thm}\label{quasi-banach opt range} Let $E$ be a quasi-Banach symmetric sequence space on $\mathbb{Z}_{+}.$ If $E(\mathbb{Z}_{+})\subset\ell_{\log}(\mathbb{Z}_{+}),$  then
\begin{enumerate}[{\rm (i)}]
\item $(F(\mathbb{Z}_{+}),\|\cdot\|_{F(\mathbb{Z}_{+})})$ is a quasi-Banach space.
\item Moreover, $(F(\mathbb{Z}_{+}),\|\cdot\|_{F(\mathbb{Z}_{+})})$ is the optimal symmetric quasi-Banach range for the operator $S^{d}$ on $E(\mathbb{Z}_{+}).$
\end{enumerate}
\end{thm}
First we need the following Lemma.
\begin{lem}\label{linearity of op range} Let $E$ be a quasi-Banach symmetric sequence space on $\mathbb{Z}_{+}.$ If $E(\mathbb{Z}_{+})\subset\ell_{\log}(\mathbb{Z}_{+}),$ then the space $F(\mathbb{Z}_{+})$ given in Definition \ref{discrete quasi-banach range} is a linear space.
\end{lem}
\begin{proof} For $j=1,2,$ let $x_{j}\in F(\mathbb{Z}_{+}),$ where $x_{j}=\{x_{j}(n)\}_{n\in\mathbb{Z}_{+}},$ and $\alpha_{j}$ be any scalars from the field of complex numbers. Then by the Definition \ref{discrete quasi-banach range} of $F(\mathbb{Z}_{+}),$ there exist corresponding $y_{j}\in E(\mathbb{Z}_{+})$ such that
 $\mu(n,x_{j})\leq (S^{d}\mu(y_{j}))(n),$ $n\in\mathbb{Z}_{+}.$ Therefore, for any $x_{j}\in F(\mathbb{Z}_{+})$ and $\alpha_{j},$  by \cite[Proposition II.1.7, p. 41]{BSh}, we have
\begin{equation}\begin{split}\label{linearity}\mu(\alpha_{1}x_{1}+\alpha_{2}x_{2})
&\leq\sigma_{2}\mu(\alpha_{1}x_{1})+\sigma_{2}\mu(\alpha_{2}x_2)=|\alpha_{1}|\cdot\sigma_{2}\mu(x_{1})+|\alpha_{2}|\cdot\sigma_{2}\mu(x_{2})\\
&\leq |\alpha_{1}|\cdot\sigma_{2}\left(S^{d}(\mu(y_{1}))\right)+|\alpha_{2}|\cdot\sigma_{2}\left(S^{d}(\mu(y_{2}))\right)\\
&= S^{d}(|\alpha_{1}|\cdot\sigma_{2}\mu(y_{1}))+S^{d}(|\alpha_{2}|\cdot\sigma_{2}\mu(y_{2}))\\
&=S^{d}(|\alpha_{1}|\cdot\sigma_{2}\mu(y_{1})+|\alpha_{2}|\cdot\sigma_{2}\mu(y_{2})).
\end{split}\end{equation}
Since $E(\mathbb{Z}_{+})$ is linear space and  $|\alpha_{1}|\cdot\sigma_{2}\mu(y_{1})+|\alpha_{2}|\cdot\sigma_{2}\mu(y_{2})\in E(\mathbb{Z}_{+})$, it follows that $\alpha_{1}x_{1}+\alpha_{2}x_{2}\in F(\mathbb{Z}_{+}).$ Thus, $F(\mathbb{Z}_{+})$ is a linear space.
\end{proof}

\begin{lem}\label{quasi-norm of op range} Let $E$ be a quasi-Banach symmetric space on $\mathbb{Z}_{+}.$ If $E(\mathbb{Z}_{+})\subset\ell_{\log}(\mathbb{Z}_{+}),$ then the space $F(\mathbb{Z}_{+})$ defined in Definition \ref{discrete quasi-banach range} is a
quasi-normed space equipped with \eqref{opt norm}.
\end{lem}
\begin{proof}Let us prove that the expression
$$\|x\|_{F(\mathbb{Z}_{+})}:=\inf\{\|y\|_{E(\mathbb{Z}_{+})}:\exists y\in E(\mathbb{Z}_{+}), \,\ \mu(x)\leq S^{d}\mu(y)\}$$
defines a quasi-norm in $F(\mathbb{Z}_{+}).$ Clearly, if $x=0,$ then, by \eqref{opt norm}, $\|x\|_{F(\mathbb{Z}_{+})}=0,$ and for every scalar $\alpha,$ we have $\|\alpha
x\|_{F(\mathbb{Z}_{+})}=|\alpha|\cdot\|x\|_{F(\mathbb{Z}_{+})}.$ We shall prove the non-trivial part. If $\|x\|_{F(\mathbb{Z}_{+})}=0,$ then there exists $y_k$ $(k\in\mathbb{N})$ in $E(\mathbb{Z}_{+})$ with $\mu(x)\leq S^{d}\mu(y_k)$ such that
$\|y_k\|_{E(\mathbb{Z}_{+})}\rightarrow0,$ as $k\rightarrow \infty.$ By assumption, we have that $S^{d}:E(\mathbb{Z}_{+})\rightarrow (\ell_{1,\infty}+\ell_{\infty})(\mathbb{Z}_{+})$ (see \ref{S dis}). Since $S^{d}$ is a positive operator (see \cite[Chapter
III, p. 134]{BSh}), it follows from \cite[Proposition 1.3.5, p. 27]{MN} that
 $S^{d}:E(\mathbb{Z}_{+})\rightarrow (\ell_{1,\infty}+\ell_{\infty})(\mathbb{Z}_{+})$ is bounded. Hence, $\|S^{d}\mu(y_k)\|_{(\ell_{1,\infty}+\ell_{\infty})(\mathbb{Z}_{+})}\rightarrow 0$ as $k\rightarrow \infty.$  From the condition $\mu(x)\leq S^{d}\mu(y_k),$ for
 every $k\in\mathbb{N},$ we have
$$\|\mu(x)\|_{(\ell_{1,\infty}+\ell_{\infty})(\mathbb{Z}_{+})}\leq\|S^{d}\mu(y_k)\|_{(\ell_{1,\infty}+\ell_{\infty})(\mathbb{Z}_{+})}\rightarrow 0,$$
which shows that $x=0.$

Let us prove that $\|\cdot\|_{F(\mathbb{Z}_{+})}$ satisfies the quasi-triangle inequality. For $j=1,2,$ let $x_{j}\in F(\mathbb{Z}_{+}),$ and fix $\varepsilon>0.$ Then there exist $y_{j}\in E(\mathbb{Z}_{+})$ with
$\mu(x_{j})\leq S^{d}\mu(y_{j})$ and $\|y_{j}\|_{E(\mathbb{Z}_{+})}<\|x_{j}\|_{F(\mathbb{Z}_{+})}+\varepsilon.$
Hence, from \eqref{opt norm} and since $\|\cdot\|_{E(\mathbb{Z}_{+})}$ is a quasi-norm, it follows from  \eqref{linearity} and \cite[Remark 18]{F}
 that
\begin{eqnarray*}\begin{split}\|x_{1}+x_{2}\|_{F(\mathbb{Z}_{+})}
&\leq\|\sigma_{2}\mu(y_{1})+\sigma_{2}\mu(y_{2})\|_{E(\mathbb{Z}_{+})}=\|\sigma_{2}\left(\mu(y_{1})+\mu(y_{2})\right)\|_{E(\mathbb{Z}_{+})}\\
&\leq2\cdot c_{E}\|\mu(y_{1})+\mu(y_{2})\|_{E(\mathbb{Z}_{+})}\leq2\cdot c^{2}_{E}\left(\|y_{1}\|_{E(\mathbb{Z}_{+})}+\|y_{2}\|_{E(\mathbb{Z}_{+})}\right)
\end{split}
\end{eqnarray*}
and by choice of $y_{1},y_{2}\in E(\mathbb{Z}_{+})$
$$\|x_{1}+x_{2}\|_{F(\mathbb{Z}_{+})}\leq2\cdot c^{2}_{E}\left(\|x_{1}\|_{F(\mathbb{Z}_{+})}+\|x_{2}\|_{F(\mathbb{Z}_{+})}\right)+4c^{2}_{E}\cdot\varepsilon.$$
Since $\varepsilon$ is arbitrary, letting $\varepsilon\rightarrow 0,$ we obtain that $\|\cdot\|_{F(\mathbb{Z}_{+})}$ defines a quasi-norm. Thus, $(F(\mathbb{Z}_{+}),\|\cdot\|_{F(\mathbb{Z}_{+})})$ is a linear quasi-normed space.
\end{proof}
Now we are ready to prove Theorem \ref{quasi-banach opt range}.
\begin{proof}[Proof of Theorem \ref{quasi-banach opt range}]
 By Lemma \ref{linearity of op range} and \ref{quasi-norm of op range}, $(F(\mathbb{Z}_{+}),\|\cdot\|_{F(\mathbb{Z}_{+})})$ is a linear quasi-normed space.

First, we prove that $F(\mathbb{Z}_{+})$ is a quasi-Banach space. To show that $F(\mathbb{Z}_{+})$ is a quasi-Banach space, it remains to see that it is complete. Since the dilation operator is bounded in any quasi-Banach symmetric space (see \cite[Proposition 2 (c)]{HM}), it follows that there is a constant $c_{E}$ depending on $E$ such that
\begin{equation}\label{dilation bound}
\|\sigma_{2^{k}}y\|_{E(\mathbb{Z}_{+})}\leq c_{E}^{k}\|y\|_{E(\mathbb{Z}_{+})}
\end{equation}
(see \cite[Remark 18]{F}) for all $y\in E(\mathbb{Z}_{+})$ and $k\in \mathbb{N}.$
On the other hand, since $E(\mathbb{Z}_{+})$ is quasi-Banach symmetric space, it follows from Aoki-Rolewicz theorem that every quasi-normed space (such as $E(\mathbb{Z}_{+}))$ is metrizable (see \cite[Theorem 1.3]{KPR}) and there exists $0<p<1,$ such that
$$\|y_{1}+y_{2}\|_{E(\mathbb{Z}_{+})}^{p}\leq\|y_{1}\|_{E(\mathbb{Z}_{+})}^{p}+\|y_{2}\|_{E(\mathbb{Z}_{+})}^{p}$$
for all $y_{1},y_{2}\in E(\mathbb{Z}_{+}).$ We have to show that an arbitrary Cauchy sequence in $F(\mathbb{Z}_{+})$
converges to an element from $F(\mathbb{Z}_{+}).$ Fix such a sequence $\{x_k\}_{k=1}^{\infty}\subset F(\mathbb{Z}_{+}).$  Without loss of generality, assume that for $\varepsilon< c_{E}^{-1},$ we have
$$\|x_{k+1}-x_{k}\|_{F(\mathbb{Z}_{+})}\leq \varepsilon^{k} $$ and
$$\mu(x_{k+1}-x_{k})\leq S^{d}\mu(y_k)$$ such that
$\|y_k\|_{E(\mathbb{Z}_{+})}\leq 2\cdot \varepsilon^{k}.$ Let us show that the series $\sum_{k=1}^{\infty}\sigma_{2^{k}}\mu(x_{k+1}-x_{k})$ converges a.e.
 Since $S^{d}$ is linear and commutes with the dilation operator,
it follows that
\begin{equation}\begin{split}\label{partial S est}
\sum_{k=1}^{\infty}\sigma_{2^{k}}\mu(x_{k+1}-x_{k})\leq\sum_{k=1}^{\infty}\sigma_{2^{k}}S^{d}\mu(y_{k})=\sum_{n=1}^{\infty}S^{d}(\sigma_{2^{n}}\mu(y_{n}))
=S^{d}\left(\sum_{k=1}^{\infty}\sigma_{2^{k}}\mu(y_{k})\right).
\end{split}\end{equation}
Hence,
\begin{eqnarray*}\begin{split}
\left\|\sum_{k=1}^{\infty}\sigma_{2^{k}}\mu(y_{k})\right\|_{E(\mathbb{Z}_{+})}^{p}
\leq\sum_{k=1}^{\infty}\|\sigma_{2^{k}}\mu(y_{k})\|_{E(\mathbb{Z}_{+})}^{p}\\
\stackrel{\eqref{dilation bound}}\leq\sum_{k=1}^{\infty}c_{E}^{kp}\|y_{k}\|_{E(\mathbb{Z}_{+})}^{p}
\leq\sum_{k=1}^{\infty}(c_{E}\cdot\varepsilon)^{kp}<\infty.
\end{split}\end{eqnarray*}
Therefore, the series $\sum_{k=1}^{\infty}\sigma_{2^{k}}\mu(y_{k})$ converges in a.e. in $E(\mathbb{Z}_{+}).$ Since $S^{d}$ is continuous on $E(\mathbb{Z}_{+})$ by assumption, it follows from \eqref{partial S est} that the series
$\sum_{k=1}^{\infty}\sigma_{2^{k}}\mu(x_{k+1}-x_{k})$ belongs to $F(\mathbb{Z}_{+}).$ Then, by Lemma \ref{measure conv}, the series $\sum_{k=1}^{\infty}(x_{k+1}-k_{n})$ converges in measure and belongs to $F(\mathbb{Z}_{+}),$
and we have
$$\|x-x_{1}\|_{F(\mathbb{Z}_{+})}=\left\|\sum_{k=1}^{\infty}(x_{k+1}-x_{k})\right\|_{F(\mathbb{Z}_{+})}\stackrel{\eqref{est conv}}\leq\left\|\sum_{k=1}^{\infty}\sigma_{2^{k}}\mu(x_{k+1}-x_{k})\right\|_{F(\mathbb{Z}_{+})}<\infty.$$
 This shows that $x\in F(\mathbb{Z}_{+}).$ So, $F(\mathbb{Z}_{+})$ is complete. On the other hand, since $\|x\|_{F(\mathbb{Z}_{+})}=\|\mu(x)\|_{F(\mathbb{Z}_{+})},$ it follows that $F(\mathbb{Z}_{+})$ is a symmetric space of sequences. So, the  space $(F(\mathbb{Z}_{+}),\|\cdot\|_{F(\mathbb{Z}_{+})})$ is a quasi-Banach symmetric sequence space.

 Next, we prove second part of the theorem. First we need to show that $S^{d}:E(\mathbb{Z}_{+})\rightarrow F(\mathbb{Z}_{+})$ is bounded. Indeed, let $x\in E(\mathbb{Z}_{+}).$ By \eqref{S proper}, we obtain
 $$\|S^{d}x\|_{F(\mathbb{Z}_{+})}=\inf\{\|y\|_{E(\mathbb{Z}_{+})}:\mu(S^dx)\leq S^d\mu(y)\}\leq\|x\|_{E(\mathbb{Z}_{+})}.$$
 Since $x\in E(\mathbb{Z}_{+})$ is arbitrary, it follows that $S^{d}:E(\mathbb{Z}_{+})\rightarrow F(\mathbb{Z}_{+})$ is bounded.

  Now, suppose that $G(\mathbb{Z}_{+})$ is another symmetric quasi-Banach sequence space such that $S^{d}:E(\mathbb{Z}_{+})\rightarrow G(\mathbb{Z}_{+})$ is bounded, and let us show that $F(\mathbb{Z}_{+})\subset G(\mathbb{Z}_{+}).$
If  $x\in F(\mathbb{Z}_{+}),$ then there is $y\in E(\mathbb{Z}_{+})$ such that $\mu(x)\leq S^{d}\mu(y).$ Hence, we have
$$\|x\|_{G(\mathbb{Z}_{+})}\leq\|S^{d}\mu(y)\|_{G(\mathbb{Z}_{+})}\leq\|S^{d}\|_{E(\mathbb{Z}_{+})\rightarrow G(\mathbb{Z}_{+})}\|y\|_{E(\mathbb{Z}_{+})}.$$
which implies via taking infimum over all such $y$'s, we obtain
$$\|x\|_{G(\mathbb{Z}_{+})}\leq\|S^{d}\|_{E(\mathbb{Z}_{+})\rightarrow G(\mathbb{Z}_{+})}\|x\|_{F(\mathbb{Z}_{+})},$$
which means $F(\mathbb{Z}_{+})\subset G(\mathbb{Z}_{+}).$ Thus, $F(\mathbb{Z}_{+})$ is the minimal space among symmetric quasi-Banach sequence spaces. This completes the proof.
\end{proof}

\begin{rmk}\label{rem1} The space $F$ given in the Definition \ref{discrete quasi-banach range} is defined similarly on $\mathbb{Z},$ and by repeating the same method as in the Theorem \ref{quasi-banach opt range} (i), it becomes a symmetric quasi-Banach space.
\end{rmk}

The following result provides a solution to the Problem \ref{dis case}.
\begin{thm}\label{opt hilbert}Let the assumptions of the Theorem \ref{quasi-banach opt range} hold. Then, so defined space $F(\mathbb{Z})$ is the optimal range for the Hilbert transform $\mathcal{H}^{d}$ on $E(\mathbb{Z}).$
\end{thm}
\begin{proof}It follows from the Theorem \ref{quasi-banach opt range} (i)  and Remark \ref{rem1} that $F(\mathbb{Z})$ is a quasi-Banach symmetric space of sequences. Let us show that $\mathcal{H}^{d}:E(\mathbb{Z})\rightarrow F(\mathbb{Z})$ is bounded. Since $S^{d}:E(\mathbb{Z})\rightarrow F(\mathbb{Z})$ is bounded by Theorem \ref{quasi-banach opt range} (ii), it follows from the discrete version of \cite[Theorem III. 4.8, p. 138]{BSh} that $\mathcal{H}^{d}:E(\mathbb{Z})\rightarrow F(\mathbb{Z})$ is bounded.
Now, suppose that $G(\mathbb{Z})$ is another symmetric quasi-Banach sequence space such that $\mathcal{H}^{d}:E(\mathbb{Z})\rightarrow G(\mathbb{Z})$ is bounded, and
let us show that $F(\mathbb{Z})\subset G(\mathbb{Z}).$
If $x\in E(\mathbb{Z}_{+}),$ then by \cite[Proposition III. 4.10, p. 140]{BSh} (here we again use discrete version) there is a function $y\in E(\mathbb{Z})$ equimeasurable with $x$ such that $S^{d}\mu(x)\leq 2\mu(\mathcal{H}^{d}y).$ Then
\begin{eqnarray*}\begin{split}\|S^{d}\mu(x)\|_{G(\mathbb{Z}_{+})}
&\leq 2\|\mu(\mathcal{H}y)\|_{G(\mathbb{Z}_{+})}=2\|\mathcal{H}y\|_{G(\mathbb{Z})}\\
&\leq c_{abs}\|y\|_{E(\mathbb{Z})}=c_{abs}\|x\|_{E(\mathbb{Z}_{+})},
\end{split}\end{eqnarray*}
which shows that
$$S^{d}:E(\mathbb{Z}_{+})\rightarrow G(\mathbb{Z}_{+})$$
is bounded. But, since $F(\mathbb{Z}_{+})$ is the least space such that $S^{d}:E(\mathbb{Z}_{+})\rightarrow F(\mathbb{Z}_{+})$ (see Theorem \ref{quasi-banach opt range} (ii)), it follows that
 $F(\mathbb{Z})\subset G(\mathbb{Z}).$ This completes the proof.
\end{proof}

A non-commutative extension of the following result was proved in \cite[Proposition 35]{STZ}.
\begin{proposition}\label{weak l1} Let $E(\mathbb{Z}_{+})=\ell_{1,\infty}(\mathbb{Z}_{+}),$ then
$$F(\mathbb{Z}_{+})=\left\{a\in (\ell_{1,\infty}+\ell_{\infty})(\mathbb{Z}_{+}):\exists c_{a}, \mu(n,a)\leq c_{a}\frac{\log(n+2)}{n+1}, \, n\in \mathbb{Z}_{+}\right\}.$$
\end{proposition}
Indeed, if $\mu(k,a)=\frac{1}{k+1}, \,\ k\in\mathbb{Z}_{+},$ then by \eqref{S dis}, we have
$$(S^{d}\mu(a))(n):=\frac{1}{n+1}\sum_{k=0}^{n}\frac{1}{k+1}+\sum_{k=n+1}^{\infty}\frac{1}{k(k+1)}$$
and
$$
(S^{d}\mu(a))(n)\approx \frac{\log(n+1)}{n+1}
$$
for large $n,$ where the symbol $\mathcal{A}\approx \mathcal{B}$ indicates that there exists universal positive constants $c_{1},c_{2}$ independent of all important parameters, such that $\mathcal{A}\leq
c_{1}\mathcal{B}$ and $\mathcal{B}\leq
c_{2}\mathcal{A}.$
Therefore, if $E(\mathbb{Z}_{+})=\ell_{1,\infty}(\mathbb{Z}_{+}),$ then the optimal range for the discrete Calder\'{o}n operator $S^{d}$ is
$$F(\mathbb{Z}_{+})=\{a\in(\ell_{1,\infty}+\ell_{\infty})(\mathbb{Z}_{+}):\exists c_{a}, \mu(n,a)\leq c_{a}\frac{\log(n+2)}{n+1}, \, n\in \mathbb{Z}_{+}\},$$
where $c_{a}$ is a constant depending only on the sequence $a.$

\begin{rmk} It follows from the Theorem \ref{opt hilbert} and previous result that if $E(\mathbb{Z})=\ell_{1,\infty}(\mathbb{Z}),$ then the optimal range for the discrete Hilbert transform $\mathcal{H}^{d}$ is
$F(\mathbb{Z}),$ which is defined as in the Proposition \ref{weak l1}.
\end{rmk}


\section{Acknowledgment}

I would like to thank Professor F. Sukochev for bringing this problem to my attention and for the support of the one year
visit to School of Mathematics and Statistics, UNSW, and the warm atmosphere at the department, where a preliminary version of the paper was written.
I also would like to thank Dr. D. Zanin for his very useful comments and helps which led to many corrections and improvements. I thank the anonymous referee for reading the paper carefully and providing thoughtful comments, which improved the exposition of the paper.
This work was partially supported by the grant No. AP05133283 of the Science Committee of the Ministry of Education and Science of the Republic of Kazakhstan.


\begin{thebibliography}{99}

\bibitem{An} K.F. Andersen, {\it Discrete Hilbert Transforms and Rearrangement invariant Sequence Spaces}, Applicable analysis, {\bf 5} (1976), 193--200.

\bibitem{BSh} C. Bennett and R. Sharpley, {\it Interpolation of Operators,} Pure and Applied Mathematics, {\bf 129}. Academic Press, 1988.


\bibitem{B} D. Boyd, {\it The Hilbert transform on rearrangement-invariant spaces}, Can. J. Math. {\bf 19} (1967), 599--616.


\bibitem{HM} H. Hudzik, L. Maligranda, {\it An interpolation theorem in symmetric function F-spaces}, Proc. Amer. Math. Soc. {\bf 110} (1) (1990), 89--96.


\bibitem{K} Y. Katznelson, {\it An introduction to Harmonic analysis}, Third Corrected Edition, Cambridge University Press, (2012).

\bibitem{KPR} N. Kalton, N. Peck, J. Rogers, {\it An F-space Sampler}, London Math. Soc. Lecture Note Ser., {\bf 89}, Cambridge University Press, Cambridge, 1985.



\bibitem{Kol} A.N. Kolmogorov, {\it Sur les fonctions harmoniques conjuguees et les series de Fourier}, Fundamenta Math. {\bf 7} (1925), 23--28.

\bibitem{Ko} Y. Komori, {\it Weak $\ell^{1}$ estimates for the generalized discrete Hilbert transforms}, Far East J. Math. Sci. {\bf 3}:2 (2001), 331--338.


\bibitem{KPS} S. Krein, Y. Petunin, and E. Semenov, {\it Interpolation of linear operators}, Amer. Math. Soc., Providence, R.I., (1982).

\bibitem{LT} J. Lindenstrauss, L. Tzafiri, {\it Classical Banach spaces}. Springer-Verlag, I, (1977).
\bibitem{LSZ} S. Lord, F. Sukochev, D. Zanin, {\it Singular traces. Theory and applications.} De Gruyter Studies in Mathematics, {\bf 46}. De Gruyter, Berlin, 2013.

\bibitem{MN} P. Meyer-Nieberg, {\it Banach Lattices}. Springer-Verlag, (1991).
\bibitem{ST} J. Soria, P. Tradacete, {\it Optimal rearrangement invariant range for Hardy-type operators }, Proc. of the Royal Soc. of Edinburgh, {\bf 146A} (2016), 865--893.

\bibitem{ST1} J. Soria, P. Tradacete, {\it Characterization of the restricted type spaces $R(X)$}, Math. Ineq. Applic. {\bf 18} (2015), 295--319.

\bibitem{F} F. Sukochev, {\it Completeness of quasi-normed symmetric operator spaces.} Indag. Math. (N.S.) {\bf 25} (2014), no. 2, 376--388.

\bibitem{STZ} F. Sukochev, K. Tulenov, and D. Zanin, {\it The optimal range of the Calder\'{o}n operator and its applications.} J. Func. Anal. (2019), 1--47. doi.org/10.1016/j.jfa.2019.05.012 (preprint)

\end{thebibliography}
\end{document}